\documentclass[11pt]{amsart}

\usepackage{amssymb,amsthm,amsmath}

\theoremstyle{plain}
\newtheorem{prop}{Proposition}[section]
\newtheorem{lem}[prop]{Lemma}

\newtheorem{thm}[prop]{Theorem}

\theoremstyle{definition}

\theoremstyle{remark}

\newcommand\N{\mathbb{N}}
\newcommand\Q{\mathbb{Q}}
\newcommand\R{\mathbb{R}}
\newcommand\Z{\mathbb{Z}}
\newcommand\C{\mathbb{C}}

\newcommand{\mc}[1]{\mathcal #1}

\DeclareMathOperator{\dist}{dist}

\begin{document}
\title[Bernstein-Walsh inequalities in higher dimensions ]{Bernstein-Walsh inequalities in higher dimensions over exponential curves}
\author{Shirali\@ Kadyrov}
\author{Mark\@ Lawrence}

\address[ML, SK]{Department of Mathematics,
Nazarbayev University,
Astana, Kazakhstan.}
\email[ML]{mlawrence@nu.edu.kz}
\email[SK]{shirali.kadyrov@nu.edu.kz}


\begin{abstract}
Let ${{\bf x}}=(x_1,\dots,x_d) \in [-1,1]^d$ be linearly independent over $\Z$, set $K=\{(e^{z},e^{x_1 z},e^{x_2 z}\dots,e^{x_d z}): |z| \le 1\}.$
We prove sharp estimates for the growth of a polynomial of degree $n$, in terms of 
$$E_n({\bf x}):=\sup\{\|P\|_{\Delta^{d+1}}:P \in \mc P_n(d+1), \|P\|_K \le 1\},$$
where $\Delta^{d+1}$ is the unit polydisk. For all ${{\bf x}} \in [-1,1]^d$ with linearly independent entries, we have the lower estimate 
$$\log E_n({\bf x})\ge \frac{n^{d+1}}{(d-1)!(d+1)} \log n - O(n^{d+1});$$ for Diophantine $\bf x$, we have $$\log E_n({\bf x})\le \frac{ n^{d+1}}{(d-1)!(d+1)}\log n+O( n^{d+1}).$$ In particular, this estimate holds for almost all $\bf x$ with respect to Lebesgue measure. The results here generalize those of \cite{CP10} for $d=1$, without relying on estimates for best approximants of rational numbers which do not hold in the vector-valued setting. 
\end{abstract}

\maketitle
\section{Introduction}

For any $\ell \in \N$ we let $\Delta^\ell$ denote the unit polydisk
$$\{{\bf z}=(z_1,z_2,\dots,z_\ell) \in \C^\ell : |z_i| \le 1, \forall i=1,2,\dots,\ell\}.$$ 
For a given $d \in \N$ we consider a vector ${{\bf x}}=(x_1,\dots,x_d) \in [-1,1]^d$ and a compact set
$$K=K({\bf x})=\{(e^{z},e^{x_1 z},e^{x_2 z}\dots,e^{x_d z}): |z| \le 1\}.$$
For any $n,\ell \in \N$ we let $\mc P_n(\ell)$ denote the subspace of polynomials $P \in \C[z_1,\dots,z_{\ell}]$ of degree $n$. For any subset $D \subset \C^{\ell}$ and polynomial $P$ we define $\|P\|_D=\{|P({\bf z})| : {\bf z} \in D\}.$ We note that $\|\cdot \|_K$ defines a norm only if $\{1,x_1,x_2,\dots,x_d\}$ are linearly independent over $\Z$ which is what we will assume throughout the paper. For any $n \in \N$ we let 
$$E_n({\bf x}):=\sup\{\|P\|_{\Delta^{d+1}}:P \in \mc P_n(d+1), \|P\|_K \le 1\}.$$
From the equivalence of the norms $\|\cdot\|_{\Delta^{d+1}}$ and $\|\cdot\|_K$ we see (c.f. \cite{CP03}) for any ${\bf z}=(z_0,z_1,\dots,z_d) \in \C^{d+1}$ that
\begin{equation}\label{eqn:BW}
|P({\bf z})| \le \|P\|_K E_n({\bf x})\exp(n\log^+\max\{|z_0|,\dots,|z_d|\}).
\end{equation}
Let $e_n({\bf x})=\log E_n({\bf x}).$ On $\R^d,$ we fix the maximum norm $\| \cdot\|$ given by $\|{\bf x}\|=\max_{1\le \ell \le  d} |x_\ell|.$ For any $x \in \R$ we let $\langle x \rangle$ denote the distance from $x$ to the nearest integer, that is, $\langle x \rangle=\min \{|x-k|: k\in \Z\}.$ We say that a vector ${\bf x} \in \R^d$ is \emph{Diophantine} if there exist $\mu \ge d$ and $\epsilon>0$ such that for any ${\bf q} \in \Z^d\backslash\{\bf 0\}$ we have $\langle {\bf q}\cdot {\bf x}\rangle >\epsilon \|{\bf q}\|^{-\mu}$. From Dirichlet's approximation theorem (see e.g. \cite{Sc80}) we know that there are no Diophantine vectors  with $\mu <d$. When $d=1$, it was shown in \cite{CP10} that if $x \in \R$ is Diophantine then the exponent $e_n(x)$ grows like $\frac12n^2 \log n$. Our goal in this paper is to generalise this result for any $d \in \N$. Using the existence of exponential polynomials in $\mathcal P_n(d+1)$ with a zero of order at least $\deg \mathcal P_n -1$ we get the following.

\begin{thm}\label{thm:genlower}
For any ${\bf x} \in \R^d$ with $\{1,x_1,\dots,x_d\}$ linearly independent over $\Z$ we have
$$e_{n}({\bf x}) \ge \frac{n^{d+1}}{(d-1)!(d+1)} \log n - O(n^{d+1}),$$
where the implied constant depends on $\bf x$ and $d$ only.
\end{thm}

As for the upper estimate we get

\begin{thm}\label{thm:main}
If ${\bf x} \in [-1,1]^d$ is Diophantine then for any $n \in \N$ we have
\begin{equation}\label{eqn:main}
e_n({\bf x}) \le \frac{ n^{d+1}}{(d-1)!(d+1)}\log n+O( n^{d+1}),
\end{equation}
where the implied constant depends on $\bf x$ and $d$ only. In particular, \eqref{eqn:main} holds for a.e.  ${\bf x} \in [-1,1]^d$.
\end{thm}

To prove their result Coman and Poletsky make use of the well developed theory of continued fractions in $\R$. As there is no good analogue of continued fractions theory in higher dimensions we will consider a different approach.

We say that a vector ${\bf x}=(x_1,\dots,x_d) \in \R^d$ with $\{x_1,\dots,x_d\}$ linearly independent over $\Q$ is \emph{Liouville} if it is not Diophantine, that is, for any $n \in \N$ there exists ${\bf q} \in \Z^d\backslash\{\bf 0\}$ such that $\langle{\bf q}\cdot {\bf x}\rangle < \|{\bf q}\|^{-n}$. Let $\mathcal L_d$ denote the set of Liouville vectors in $\R^d$. Let $W_d(\alpha)$ denote the set of vectors ${\bf x} \in \R^d$ such that there are infinitely many integer vectors ${\bf q} \in \Z^d$ satisfying $\langle {\bf q} \cdot {\bf x}\rangle < \|{\bf q}\|^{-\alpha}$. It was proved in \cite{BD86} that the Haudorff dimension of $W_d(\alpha)$ is $(d-1)+\frac{d+1}{1+\alpha}.$ Since $\mathcal L_d=\cap_{\alpha \ge d} W_d(\alpha)$, it follows that the Hausdorff dimension of $\mathcal L_d$ is at most $d-1.$ In particular, $\mathcal L_d$ has zero Lebesgue measure which justifies the last part of Theorem~\ref{thm:main}.

We note that for any nonzero ${\bf q} \in \Z^d$ the set $\{{\bf x} \in \R^d : {\bf q}\cdot {\bf x}=0 \}$ is a hyperplane in $\R^d$ and is contained in $\mathcal L_d$. Together with the above upper estimate we get that the set $\mathcal L_d$ of Liouville $d$-vectors has Hausdorff dimension $d-1$.

We now turn to discuss the exceptional set of points in $\R^d$ for which $e_n({\bf x})$ grows faster than $C n^{d+1} \log n$. To this end, we define the set
$$W(d)=\left\{{\bf x} \in [-1,1]^d : \limsup_{\|{\bf q}\| \to \infty} \frac{-\log \langle {\bf q} \cdot {\bf x}\rangle}{\|{\bf q}\|^{d+1}\log \|{\bf q}\|} =\infty \right\},$$
where ${\bf q} \in \Z^d_{\ge 0}:=\{(q_1,\dots,q_d) \in \Z^d : q_1,\dots,q_d \ge 0\}$.

\begin{thm}\label{thm:lowerfast}
For any ${\bf x}\in W(d)$, $\limsup_n \frac{e_n({\bf x})}{n^{d+1}\log n}=\infty.$ 
\end{thm}

It is easy to see (e.g. from Theorem~\ref{thm:main}) that $W(d) \subset \mathcal L_d$ so that it has Hausdorff dimension at most $d-1$. In fact, we have

\begin{thm} \label{thm:dim}
Hausdorff dimension of the exceptional set $W(d)$ is $d-1$. 
\end{thm}

It was proved in \cite{CP10} that when $d=1$ the set of points ${x}$ for which $e_n(x)$ grow faster than $\frac12n^2\log n$ is uncountable. For $d >1$, since the Hausdorff dimension of $W(d)$ is positive we in particular get that $W(d)$ is uncountable. Thus, for any $d \in \N$ the set of points ${\bf x}$ for which $e_n({\bf x})$ grow faster than $\frac1{(d-1)!(d+1)} n^{d+1}\log n$ is uncountable and has Hausdorff dimension $d-1$.

In the next section we will prove Theorem~\ref{thm:main} and in \S~\ref{sec:lower} we obtain Theorem~\ref{thm:genlower}, Theorem~\ref{thm:lowerfast}, and Theorem~\ref{thm:dim}.

\subsection*{Acknowledgement}
The authors are grateful to Dan Coman for useful comments in the preliminary version of the paper.

\section{Upper estimate}
In this section our goal is to obtain Theorem~\ref{thm:main}. We state \cite[Lemma 2.4]{CP10}

\begin{lem}\label{lem:dist} Let $x,y \in \Z$ with $x \le y$ be given. For any $\alpha \in \R$ we have
\[ \prod_{j = x}^y |j-\alpha| \ge \langle \alpha \rangle
      \left(\frac{y-x}{2e}\right)^{y-x} . 
      \]\end{lem}

Let ${\bf x} \in \R^d$ and $n\in \N$ be given. For any $ \ell \in \{0,1,\dots,n\}$ and ${\bf m} \in \Z^d$ with $m_1,\dots,m_d \in \{0,1,\dots,n\}$ we define
\begin{equation}\label{eqn:beta}
\beta(\ell,{\bf m})= \prod_{j_0+j_1+\cdots+j_d \le n, (j_0,{\bf j})\ne (\ell,{\bf  m})} ((\ell- j_0)+({\bf m-j})\cdot {\bf x}),
\end{equation}
where each ${\bf j}=(j_1,\dots,j_d) \in \Z^d$ has nonnegative components.
We will need the following estimate.

\begin{prop}\label{lem:beta}
If ${\bf x}$ is Diophantine, then there exists a constant $C_{{\bf x},d}>0$ such that 
$$\log |\beta(\ell,{\bf m})| \ge \frac{1}{(d+1)!}n^{d+1}\log n- C_{{\bf x},d} n^{d+1}.$$
\end{prop}

To obtain the proposition we need the following lemmas. We set $|{\bf j}|=j_1+\cdots+j_d.$ Arguing inductively on $d$ it is easy to see that

\begin{lem}\label{lem:md1}
For any $m \in \N$, the set $\left\{{\bf j} \in \Z^d : |{\bf j}|=m, j_1,\dots,j_d \ge 0\right\}$ has cardinality 
$C(m+d-1,d-1)=\left( \begin{array}{c}  m+d-1\\d-1  \end{array} \right).$
\end{lem}

\begin{lem}\label{lem:integralest}
We have
$$\int_1^n (n-x)^{d-1}x \log x \,dx \ge \frac{1}{d(d+1)} n^{d+1}\log n - C_d n^{d+1}.$$
\end{lem}

\begin{proof}
We claim for any $m,\ell \ge 1$ that 
$$\int_1^n (n-x)^m x^\ell \log x\, dx\ge \frac{m}{\ell+1}\left[ \int_1^n (n-x)^{m-1}x^{\ell+1}\log x\,dx-\frac{n^{m+\ell+1}}{\ell+1} \right].$$
We first note from integration by parts that 
$$\int x^\ell \log x\,dx=\frac{x^{\ell+1}}{\ell+1}\log x-\int \frac{x^{\ell}}{\ell+1} dx=\frac{x^{\ell+1}}{\ell+1}\log x-\frac{x^{\ell+1}}{(\ell+1)^2}+C.$$
Now, using integration by parts again we obtain:
\begin{multline*}
\int_1^n (n-x)^m x^\ell \log x\, dx=(n-x)^m \left(\frac{x^{\ell+1}}{\ell+1}\log x-\frac{x^{\ell+1}}{(\ell+1)^2}\right)\big\vert_1^n\\
+\int_1^n m(n-x)^{m-1} \left(\frac{x^{\ell+1}}{\ell+1}\log x-\frac{x^{\ell+1}}{(\ell+1)^2}\right)\,dx.
\end{multline*}
We note that $(n-x)^{m-1}x^{\ell+1} \le n^{m+\ell}$ for $x \in [1,n]$. Thus, simplifying we get
\begin{align*}
\int_1^n (n-x)^m x^\ell \log x\, dx&\ge  \frac{(n-1)^m}{(\ell+1)^2}+\frac{m}{\ell+1}\int_1^n \left[(n-x)^{m-1} x^{\ell+1}\log x -\frac{n^{m+\ell}}{\ell+1}\right]\,dx\\
&\ge \frac{m}{\ell+1}\left[ \int_1^n (n-x)^{m-1}x^{\ell+1}\log x\,dx-\frac{n^{m+\ell+1}}{\ell+1} \right].
\end{align*}
 To prove the lemma we iterate the claim:
\begin{align*}
\int_1^n (n-x)^{d-1}x \log x \,dx &\ge \frac{d-1}{2}\left[ \int_1^n (n-x)^{d-2}x^{2}\log x\,dx-\frac{ n^{d+1}}{2}\right]\\
& \ge \frac{d-1}{2} \left[ \frac{d-2}{3}\left(\int_1^n (n-x)^{d-3}x^{3}\log x\,dx-\frac{n^{d+1}}3 \right)-\frac{n^{d+1}}2\right]\\
&\ge \cdots\\
&\ge \frac{(d-1)!}{d!} \int_1^n x^d \log x\,dx - C_d' n^{d+1}\\
&= \frac{1}{d(d+1)} n^{d+1}\log n - C_d n^{d+1}. \qedhere
\end{align*}
\end{proof}

We state without proof the following 
\begin{lem}\label{lem:locmax}
Let $m <n$ be integers and $f:[m,n] \to [0,\infty)$ be a continuous function with exactly one local maximum in $[m,n]$ and $f(m)=f(n)=0$. Then, we have
\begin{equation}
 \left|\sum_{k=m}^n f(k) -\int_m^n f(x) \,dx\right| \le \max_{m\le x \le n} f(x).
\end{equation}
\end{lem}

\begin{proof}[Proof of Proposition~\ref{lem:beta}]

We have
$$|\beta(\ell,{\bf m})|\ge \prod_{|{\bf j}| \le n,\, {\bf j}\ne {\bf  m}} \prod_{j_0=0}^{n-|{\bf j}|}|(\ell- j_0)+({\bf m-j})\cdot {\bf x}|.$$
Since ${\bf x}$ is Diophantine of order $\mu$ we may find some $\epsilon>0$ such that $\langle {\bf q}\cdot {\bf x}\rangle \ge \epsilon\|{\bf q}\|^{-\mu}$. Using Lemma~\ref{lem:dist}  we get
\begin{align*}
|\beta(\ell,{\bf m})|&\ge \prod_{|{\bf j}|\le n,\, {\bf j}\ne {\bf  m}} \prod_{j=-\ell}^{n-|{\bf j}|-\ell}|j-({\bf m-j})\cdot {\bf x}|\\
& \ge \prod_{|{\bf j}|\le n,\, {\bf j}\ne {\bf  m}} \left( \frac{n-|{\bf j}|}{2e}\right)^{n-|{\bf j}|}\langle ({\bf m-j})\cdot {\bf x}\rangle \\
& \ge \prod_{|{\bf j}|\le n,\, {\bf j}\ne {\bf  m}} \left( \frac{n-|{\bf j}|}{2e}\right)^{n-|{\bf j}|} \prod_{|{\bf j}|\le n,\, {\bf j}\ne {\bf  m}}\epsilon \|{\bf m-j}\|^{-\mu}\\
& =\left(\prod_{k=1}^n \prod_{|{\bf j}|=n-k,\, {\bf j}\ne {\bf  m}} \left( \frac{k}{2e}\right)^{k} \right)\left(\prod_{|{\bf j}|\le n,\, {\bf j}\ne {\bf  m}} \epsilon \|{\bf m-j}\|^{-\mu}\right).
\end{align*}
We set 
$$A:=\prod_{k=1}^n \prod_{|{\bf j}|=n-k,\, {\bf j}\ne {\bf  m}} k^{k}, B:=\prod_{k=1}^n \prod_{|{\bf j}|=n-k,\, {\bf j}\ne {\bf  m}} (2e)^{-k}, C:=\prod_{|{\bf j}|\le n,\, {\bf j}\ne {\bf  m}} \epsilon \|{\bf m-j}\|^{-\mu}.$$
We now estimate each of $A,B,C$ separately.
Since the set $\{{\bf j} \in \Z^{d}: |{\bf j}| \le n\}$ has cardinality at most $(n+1)^d$ and $\|{\bf m-j}\| \le n$ for any $|{\bf j}| \le n$ we get that 
$$C=\prod_{|{\bf j}|\le n,\, {\bf j}\ne {\bf  m}}\epsilon \|{\bf m-j}\|^{-\mu}\ge \prod_{|{\bf j}|\le n} \epsilon n^{-\mu}\ge \epsilon^{(n+1)^d} n^{-\mu (n+1)^d} \ge  \epsilon^{(2n)^d} n^{-\mu (2n)^d}.$$
Thus,
\begin{equation}\label{eqn:stir}
\log C \ge -\mu 2^d n^d \log n+2^d n^d \log \epsilon.
\end{equation}
Using Lemma~\ref{lem:md1} together with the trivial bound we get
\begin{align*}
\log A &\ge \left(\sum_{k=1}^n \sum_{|{\bf j}|=n-k}k\log k\right) - n\log n\\&=\left(\sum_{k=1}^n\left( \begin{array}{c}  n-k+d-1\\d-1  \end{array} \right) k\log k\right) - n\log n\\
& \ge \left(\frac{1}{(d-1)!}\sum_{k=1}^n(n-k)^{d-1} k \log k\right) - n \log n.
\end{align*}
It is easy to see that the function $f:[1,n] \to [0,\infty)$ given by $f(x)=(n-x)^{d-1}x \log x$ satisfies Lemma~\ref{lem:locmax}. Thus, Lemma~\ref{lem:integralest} and Lemma~\ref{lem:locmax} give 
\begin{align*}
\log A &\ge  \frac{1}{(d-1)!}\left(\int_{1}^n(n-x)^{d-1} x \log x \, dx - \max_{1 \le x \le n}f(x)\right) - n \log n\\
& \ge   \frac{1}{(d-1)!} \left(  \frac{1}{d(d+1)} n^{d+1}\log n - C_d n^{d+1}-n^d \log n\right) -n\log n.
\end{align*}
Hence,
\begin{equation}\label{eqn:A}
\log A\ge \frac{1}{(d+1)!} n^{d+1}\log n - 3C_d n^{d+1}.
\end{equation}
On the other hand, since $C(n-k+d-1,d-1) \le \frac{n^{d-1}}{(d-1)!}+O(n^{d-2})$ for any $k \in [1,n]$ we get
\begin{multline}\label{eqn:B}
\log B \ge -\sum_{k=1}^n \sum_{|{\bf j}|=n-k}k\log (2e) =-\sum_{k=1}^n\left( \begin{array}{c}  n-k+d-1\\d-1  \end{array} \right) k\log (2e) \\
 \ge -\frac{2}{(d-1)!}n^{d+1} - O(n^d),
\end{multline}
where the implied constant depends $d$ only. Thus, combining \eqref{eqn:stir}, \eqref{eqn:A} and \eqref{eqn:B} we arrive at
\begin{equation*}
\log |\beta(\ell,{\bf m})| >  \frac{1}{(d+1)!} n^{d+1}\log n - C_{d,\mu,\epsilon} n^{d+1}. \qedhere
\end{equation*}
\end{proof}

\begin{proof}[Proof of Theorem~\ref{thm:main}]
Let $N=\dim\mc P_n -1$, so that $N=\left( {\begin{array}{c} n+d+1 \\ n  \end{array} } \right)-1.$

Fix some $P \in \mc P_n$ with $\|P\|_K \le 1.$ Define
$$P({\bf z})=\sum_{j_0+j_1+\cdots+j_d \le n} c(j_0,{\bf j}) z_0^{j_0}\cdots z_d^{j_d} \text{ and } f(z)=P(e^z,e^{x_1 z},\dots, e^{x_d z}),$$
where $j_0,\dots,j_d \ge 0.$ Then,
 $$f(z)=\sum_{j_0+j_1+\cdots+j_d \le n} c(j_0,{\bf j}) e^{(j_0+{\bf j}\cdot {\bf x})z}.$$
For any polynomial $R(\lambda)=\sum_{j=0}^m c_j \lambda^j$ we introduce the differential operator
$$D_R=R\left(\frac{d}{dz}\right)=\sum_{j=0}^m c_j\frac{d^j}{dz^j}.$$
We note that for any $a \in \C$ we have
\begin{equation}\label{eqn:oper}
D_R(e^{az}) \mid_{z=0}=\sum_{j=0}^mc_j a^j = R(a).
\end{equation}
To estimate $c(\ell,{\bf m})$ we set
$$R_{\ell,{\bf m}}(\lambda)=\prod_{j_0+j_1+\cdots+j_d \le n, (j_0,{\bf j})\ne (\ell,{\bf  m})} (\lambda- (j_0+{\bf j}\cdot {\bf x}))=\sum_{t=0}^Na_t\lambda^t.$$
For any $\lambda \ge 0 $ we have
$$\sum_{t=0}^N|a_t|\lambda^t \le \prod_{j_0+j_1+\cdots+j_d \le n, (j_0,{\bf j})\ne (\ell,{\bf  m})} (\lambda+ |j_0+{\bf j}\cdot {\bf x}|) \le (\lambda+n)^N.$$
From \eqref{eqn:oper} we note that 
$$D_{R_{\ell,{\bf m}}}(e^{(j_0+{\bf j}\cdot {\bf x})z})\mid_{z=0}=\begin{cases} 
      R_{\ell,{\bf m}}(\ell+{\bf m}\cdot {\bf x}) & \textrm{ if $(j_0,{\bf j})= (\ell,{\bf  m})$}, \\
      0 & \textrm{ if $(j_0,{\bf j})\ne (\ell,{\bf  m})$}. \\
   \end{cases} $$
Thus,
$$D_{R_{\ell,{\bf m}}}(f(z))\mid_{z=0}=c(\ell,{\bf m})\beta(\ell,{\bf m})$$
where $\beta(\ell,{\bf m})$ from \eqref{eqn:beta}.

On the other hand, using $\|P\|_K \le 1$ and Cauchy's inequality we get
\begin{equation}
|f^{(t)}(0)| \le t! \le N^t \text{ whenever }t \le N.
\end{equation}
This implies that
$$\left|D_{R_{\ell,{\bf m}}}(f(z))\mid_{z=0}\right|=\left|\sum_{t=0}^N a_t f^{(t)}(0)\right| \le \sum_{t=0}^N |a_t|N^t \le (N+n)^N.$$
Therefore,
$$\log(|c(\ell,{\bf m})\beta(\ell,{\bf m})|) \le N\log(N+n).$$
Using Proposition~\ref{lem:beta} we obtain
\begin{multline*}
\log(|c(\ell,{\bf m})|) \le N\log(N+n) - \log |\beta(\ell,{\bf m})|\\
 \le  N\log(N+n) - \frac1{(d+1)!}n^{d+1}\log n + C_{{\bf x},d} n^{d+1}.
\end{multline*}
Since $\|P\|_{\Delta^d} \le \sum |c(j_0,{\bf j})| \le (N+1)\max  |c(j_0,{\bf j})|$ we deduce that
\begin{align*}
e_n({\bf x}) &\le  N\log(N+n) - \frac1{(d+1)!}n^{d+1}\log n + C_{{\bf x},d} n^{d+1}+\log (N+1).
\end{align*}
Finally, using
\begin{equation}\label{eqn:N}
N =C(n+d+1,d+1)-1= \frac{n^{d+1}}{(d+1)!}+O(n^d)
\end{equation}
we obtain $N\log(N+n) \le N\log N +N =\frac{1}{d!}n^{d+1}\log n+O(n^{d+1})$. Hence,
$$e_n({\bf x}) \le \frac{n^{d+1}}{(d-1)!(d+1)} \log n+O( n^{d+1}). \qedhere$$
\end{proof}

\section{Lower estimate and Hausdorff dimension}\label{sec:lower}
We first start proving Theorem~\ref{thm:genlower}. It is essentially contained in the proof of \cite[Proposition~1.3]{CP03} as pointed out by D.~Coman and for completeness we recall it here.

\begin{proof}[Proof of Theorem~\ref{thm:genlower}]
Fix $P \in \mathcal P_n(d+1)$ with $ ord(P(e^z,e^{x_1 z},\dots,e^{x_d z}),0) \ge N$. We have $P \not \equiv 0$ implies $P(e^z,e^{x_1 z},\dots,e^{x_d z}) \not \equiv 0$. We let $f(z)=\frac{1}{\|P\|_K} P(e^z,e^{x_1 z},\dots,e^{x_d z})$ so that $\|f\|_{\Delta^{d+1}}=1$ then $\max_{|z|=r}|f(z)| \ge r^N, r \ge 1.$ From \eqref{eqn:BW} we get for any $|z|=r$
$$r^N \le  E_n({\bf x})\exp(n\log^+\max\{|e^z|,|e^{x_1 z}|,\dots,|e^{x_d z}|\}) \le E_n({\bf x})e^{nC_0 r},$$
where $C_0=\max\{1,\|{\bf x}\|\}$. Taking $r=N/n$ we see that
$$N\log \frac{N}n \le e_n({\bf x})+C_0 N.$$ 
Using \eqref{eqn:N} we have
$$N\log \frac{N}n = \frac{n^{d+1}}{(d-1)!(d+1)}\log n+O(n^d\log n),$$
which gives
$$e_n({\bf x}) \ge \frac{n^{d+1}}{(d-1)!(d+1)}\log n -O(n^{d+1}).\qedhere$$
\end{proof}

Now we prove Thoerem~\ref{thm:lowerfast} which provides us with the exceptional set of points ${\bf x}$ that does not satisfy Theorem~\ref{thm:main}.

\begin{proof}[Proof of Theorem~\ref{thm:lowerfast}]
Let ${\bf x} \in W(d)$ and $({\bf q}_\ell)_{\ge 1}$ be a sequence satisfying 
\begin{equation}\label{eqn:qell}
C(\ell)=\frac{-\log \langle {\bf q}_\ell \cdot {\bf x}\rangle}{\|\bf q_\ell\|^{d+1}\log \|{\bf q_\ell}\|} \to \infty \text{ as } \ell \to \infty.
\end{equation}
For a given $\ell \ge 0$ we let $n = d\|{\bf q}_\ell\|  $ and $p \in \Z$ be such that $\langle {\bf q_\ell}\cdot{\bf x} \rangle = |{\bf q_\ell}\cdot{\bf x} -p |$. Since $\|{\bf x}\| \le 1$ we see that $|p| \le d \|{\bf q}_\ell\|$. Then, the polynomial $P$ given by
$$P(z_0,z_1,\dots,z_d)=z_0^p-\prod_{\ell=1}^d z_\ell^{q_\ell}$$
is in $\mathcal P_n(d+1)$.
Clearly, $\|P\|_{\Delta^{d+1}}=2$. Using $|1-e^\xi| \le 2 |\xi|$ for $|\xi| \le 1$  we get
$$
|P(e^z,e^{x_1 z},\dots,e^{x_d z})|=|e^{pz}(1-e^{({\bf q}_\ell\cdot {\bf x}-p)z})| \le 2 e^{n}\langle {\bf q}_\ell \cdot {\bf x}\rangle,
$$
whenever $|z| \le 1$. Therefore, 
$$E_{n}({\bf x}) \ge \|P\|_{\Delta^{d+1}} / \|P\|_K \ge   e^{-n}\frac1{\langle {\bf q}_\ell \cdot {\bf x}\rangle} .$$
So, using \eqref{eqn:qell} we get 
\begin{align*}
e_n({\bf x})=\log E_{n}({\bf x}) &\ge C(\ell)\|{\bf q}_\ell\|^{d+1} \log \|{\bf q}_{\ell}\|-n\\
&= C(\ell)\left(\frac{n}{d}\right)^{d+1} \log \frac{n}{d} -n.
\end{align*}
Thus,
$$\frac{e_n({\bf x})}{n^{d+1}\log n}  \ge \frac{1}{d^{d+1}}C(\ell)-\frac1n \to \infty \text{ as } \ell \to \infty.\qedhere$$
\end{proof}

It remains to give the proof of Theorem~\ref{thm:dim}.

\begin{proof}[Proof of Theorem~\ref{thm:dim}]

We will use ubiquitous systems introduced in \cite{DRV90} as a method of computing Hausdorff dimension of lim-sup sets. We consider the family $\mathcal R=\{R({\bf q}): {\bf q} \in \Z^d_{\ge 0} \}$ where for any ${\bf q} \in \Z^d$ we set $R({\bf q}):=\{{\bf x} \in \R^d : {\bf q}\cdot {\bf x} \in \Z\}.$
Let $\psi: \N \to [0,1]$ be a decreasing function converging to $0$ at the infinity. Define
$$\Lambda(\mathcal R; \psi)=\left\{{\bf x} \in [-1,1]^d : \dist({\bf x}, R({\bf q}))<\psi(\|{\bf q}\|) \text{ for infinitely many } R({\bf q})\right\},$$
 where $\dist({\bf x},S)=\inf_{{\bf y} \in S} \|{\bf x}-{\bf y}\|$. 
For any such $\psi$, we will prove that the Hausdorff dimension of $\Lambda(\mathcal R; \psi)$ is at least $d-1$. Then, for $\psi(n)=n^{-n^{d+2}}$ we will show that $\Lambda(\mathcal R; \psi) \subset W(d)$ which will finish the proof.

Let $I^d$ denote the hypercube $[-\frac12,\frac12]^d$ of unit length. It is well-known (see e.g.\cite{Do93}) that the family $\{R({\bf q}): {\bf q} \in \Z^d \}$ is \emph{ubiquitous with respect to $\rho(Q):=dQ^{-1-d}\log Q$} in the sense that
$$\left| I^d\backslash \bigcup_{1 \le \|{\bf q}\| \le N} B(R({\bf q});\delta(N)) \right| \to 0 \text{ as } N \to \infty,$$
 $B(R({\bf q});\delta)=\{{\bf x} \in \R^d : \dist({\bf x},R({\bf q}))<\delta\}.$
However, it is not clear if the family $\mathcal R=\{R({\bf q}): {\bf q} \in \Z^d_{\ge 0} \}$ is ubiquitous with respect to the same $\rho$. However, for our purposes we do not need to try to optimize $\rho$. Simply consider the constant function $\rho \equiv 1$, then for ${\bf q}=(0,\dots,0,1)$ we have $I^d \subset B(R({\bf q});1)$ so that $\mathcal R$ is ubiquitous w.r.t 1.  Since
$\gamma:=\limsup_{Q \to \infty} \frac{\log \rho(Q)}{\log \psi(Q)}=0,$
it follows from \cite[Theorem 1]{DRV90} that the Hausdorff dimension of $\Lambda(\mathcal R; \psi)$ is at least $\dim \mathcal R + \gamma\, {\rm codim  } \,\mathcal R=d-1$.

We now claim that $\Lambda(\mathcal R; \psi) \subset W(d)$ when $\psi(n)=n^{-n^{d+2}}$. For ${\bf x} \in \Lambda(\mathcal R; \psi)$ let $({\bf q}_\ell)_{\ge \ell}$ denote the sequence such that $ \dist({\bf x}, R({\bf q}_\ell))<\psi(\|{\bf q}\|)$ and all $R({\bf q}_\ell)$ are distinct. Then, for any ${\bf q} \in \Z^d$ we have ${\bf q} \cdot {\bf x} \not \in \Z$ which means $\{1,x_1,x_2,\dots,x_d\}$ is linearly independent over $\Z$. Let ${\bf y} \in R({\bf q}_\ell)$ be such that $ \|{\bf x}-{\bf y}\|<\psi(\|{\bf q}\|)$. We choose $p \in \Z$ with ${\bf q}_\ell\cdot {\bf y}=p$. Then,
$$\langle {\bf q_\ell} \cdot {\bf x} \rangle \le \| {\bf q_\ell} \cdot ({\bf x}-{\bf y})+ {\bf q_\ell} \cdot {\bf y}-p\| \le \| {\bf q_\ell}\|\| {\bf x}-{\bf y}\|< \psi(\|{\bf q}_\ell\|).$$
 Hence, ${\bf x} \in W(d)$ as $\frac{-\log \langle {\bf q}_\ell \cdot {\bf x}\rangle}{\|{\bf q}_\ell\|^{d+1}\log \|{\bf q}_\ell\|} \ge \|{\bf q}_\ell\|$ and $ \|{\bf q}_\ell\| \to \infty$ with $\ell \to \infty$.
\end{proof}

\end{document}